\documentclass[a4paper,reqno,11pt]{article}
\usepackage[a4paper, margin=1in]{geometry}
\usepackage[american]{babel}
\usepackage{upref}
\usepackage{varioref}
\usepackage{enumerate}
\usepackage[ansinew]{inputenc}
\usepackage{mathtools}
\usepackage{amssymb}
\usepackage{amsmath}
\usepackage{amsthm}
\usepackage[colorlinks,
  pdftitle={Vertex-coloring graphs with 4-edge-weightings},
  pdfview=FitH]{hyperref}
\allowdisplaybreaks[2]

\newtheorem{theorem}{Theorem}

\newtheorem{lemma}[theorem]{Lemma}

\newcommand{\NN}{\mathbb N}

\newcommand{\Gnp}{\mathbb{G}(n,p)}

\renewcommand{\epsilon}{\varepsilon}

\begin{document}
\title{Vertex-coloring graphs with 4-edge-weightings}
\author{Ralph Keusch \\ \small{ralphkeusch@gmail.com}}
\date{\today}
\maketitle
\begin{abstract}
An edge-weighting of a graph is called vertex-coloring if the weighted degrees yield a proper vertex coloring of the graph. It is conjectured that for every graph without isolated edge, a vertex-coloring edge-weighting with the set $\{1,2,3\}$ exists. In this note, we show that the statement is true for the weight set $\{1,2,3,4\}$.
\end{abstract}

\section{Introduction}
Let $G=(V,E)$ be a simple graph. A $k$-edge-weighting is a function $\omega: E \rightarrow \{1, \ldots, k\}$. Given an edge-weighting $\omega$, for a vertex $v \in V$, we denote by $s_{\omega}(v)=\sum_{w \in N(v)} \omega(\{v,w\})$ its weighted degree. We say that $\omega$ is vertex-coloring if for each edge $e=\{u,v\} \in E$, it holds $s_{\omega}(u) \neq s_{\omega}(v)$. Obviously, if $G$ contains an isolated edge, no edge-weighting is vertex-coloring. Otherwise, we aim to find a vertex-coloring $k$-edge-weighting with the smallest possible integer $k$. In 2004, Karo\'{n}ski, {\L}uczak, and Thomason conjectured that for each graph without connected component isomorphic to $K_2$, a vertex-coloring $3$-edge-weighting exists \cite{karonski2004edge}. This statement is also known as $1$-$2$-$3$-conjecture. For instance, if $G$ is a cycle of length not divisible by $4$, no $2$-edge-weighting is vertex-coloring, thus $k=3$ is best possible in general.

Addario-Berry, Dalal, McDiarmid, Reed, and Thomason proved the first general upper bound of $k=30$ \cite{addarioberry2007vertexcolouring}. The bound was improved to $k=16$ by Addario-Berry, Dalal, and Reed \cite{addarioberry2008degree}, to $k=13$ by Wang and Yu \cite{wang2008onvertex}, and then, in a significant improvement, to $k=5$ by Kalkowski, Karo\'{n}ski, and Pfender \cite{kalkowski2010vertexcoloring} in 2010. For a random graph $\Gnp$, asymptotically almost surely there exists a vertex-coloring $2$-edge-weighting \cite{addarioberry2008degree}. For $d$-regular graphs, an upper bound of $k=4$ holds in general \cite{przybylo2021theconjecture, bensmail2018result} and the conjecture (i.e., an upper bound of $k=3$) is confirmed for $d>10^8$ \cite{przybylo2021theconjecture}. Recently, Przyby{\l}o verified the conjecture for graphs where the minimum degree is sufficiently large compared to the maximum degree \cite{przybylo2020conjecture}. Furthermore, the conjecture was confirmed for $3$-colorable graphs \cite{karonski2004edge} and for dense graphs \cite{zhong2018theconjecture}. Moreover, Dudek and Wajc proved that it is NP-complete to decide whether a given graph $G$ admits a vertex-coloring $2$-edge-weighting \cite{dudek2011complexity}. For an early overview on the conjecture and on related problems, we refer to the survey of Seamone \cite{seamone2012theconjecture}.

The contribution of this note is an improved general upper bound of $k=4$ with the following result.

\begin{theorem}\label{thm:main}Let $G=(V,E)$ be a graph without connected component isomorphic to $K_2$. Then there exists an edge-weighting $\omega:E \rightarrow \{1,2,3,4\}$ such that for any two neighbors $u$ and $v$,
$$\sum_{w \in N(u)}\omega(\{u,w\}) \neq \sum_{w \in N(v)}\omega(\{v,w\}).$$
\end{theorem}

\section{Proof}
We start by introducing notation, giving a high-level overview of the proof, and collecting two auxiliary results. Let $G=(V,E)$ be a graph and let $C=(S,T)$ be a cut. We denote by $E(S)$ the edge set of the induced subgraph $G[S]$ and by $E(S,T)$ the subset of edges having an endpoint in both $S$ and $T$ (the cut edges of $C$). For a vertex $v \in V$, we denote by $N(v)$ its neighborhood.

We will start the proof by identifying a vertex $v_0$ that is handled separately. Next, we will take a maximum cut $C=(S,T)$ of $G[V \setminus \{v_0\}]$ and construct an edge-weighting such that the weighted degree of a vertex $v \neq v_0$ is even if $v \in S$ and odd if $v \in T$. There may still be conflicts, that is, neighboring vertices with the same weighted degree. The main idea to solve these conflicts is to modify the edge-weighting along sufficiently many edge-disjoint paths that can be found with the subsequent lemma.

\begin{lemma}\label{lemma:flow}
Let $G=(V,E)$ be a graph and let $C=(S,T)$ be a maximum cut of $G$. Furthermore, let $F\subseteq E(S) \cup E(T)$  and let $\sigma$ be an orientation of the edge set $F$. Let $G_{C, F, \sigma}$ be the auxiliary directed multigraph network constructed as follows.
\begin{enumerate}[(i)]
\item As vertex set, take $V$, and add a source node $s$ and a sink node $t$.
\item For each edge $\{u,v\} \in E(S,T)$, insert the two arcs $(u,v)$ and $(v,u)$, both with capacity $1$.
\item For each edge $\{u,v\} \in F$ with corresponding orientation $(u,v) \in \sigma$, insert arcs $(s, u)$ and $(v, t)$, both with capacity $1$, potentially creating multi-arcs. Do not insert $(u,v)$.
\end{enumerate}
Then in the network $G_{C, F, \sigma}$, there exists an $s$-$t$-flow of size $|F|$.
\end{lemma}

The cardinalities of $S$ and $T$ will determine the parity of the weighted degree of the remaining vertex $v_0$. It remains to ensure that the edge-weighting also properly colors $v_0$, which will be done by applying the following lemma to the induced subgraph $G[N(v_0) \cup \{v_0\}]$ and using especially property (iv) of its statement.

\begin{lemma}\label{lemma:singlevertex}
Let $G=(V,E)$ be a graph with $|V| \ge 3$, let $v_0 \in V$ such that $\deg(v_0) = |V|-1$, and let $g: V \rightarrow \NN_0$ such that for each edge $\{u,v\} \in E(N(v_0))$, it holds $g(u) \neq g(v)$. Then there exists a function $h: E \rightarrow \{0, 1, 2\}$ such that
\begin{enumerate}[(i)]
\item $h(\{u,v\}) \in \{0,1\}$, whenever $v_0 \notin \{u,v\}$ and $g(u)+ g(v)$ is even,
\item $h(\{u,v\})=0$, whenever  $v_0 \notin \{u,v\}$ and $g(u) + g(v)$ is odd,
\item $s_h(v) := \sum_{w \in N(v)} h(\{v,w\}) \in \{0,2\}$ for all $v \in N(v_0)$, and
\item $g(u)+s_h(u) \neq g(v) + s_h(v)$ for each edge $\{u,v\} \in E$.
\end{enumerate}
\end{lemma}

We now start proving the theorem. The proofs of Lemma~\ref{lemma:flow} and Lemma~\ref{lemma:singlevertex} are deferred to the end of the paper.

\begin{proof}[Proof of Theorem~\ref{thm:main}]Assume w.l.o.g.\ that $G$ is connected and has at least three vertices. We give each edge $e$ the provisorial weight $\mu(e)=2$, which will be  modified later on. Denote by $s_{\mu}(v):=\sum_{w \in N(v)}\mu(\{v,w\})$ the weighted degree of a vertex $v \in V$ under $\mu$. Let $v_0$ be a vertex which is not an articulation node (for instance, take a leaf of a spanning tree), so that the reduced graph $H := G[V \setminus \{v_0\}]$ is still connected. Next, take a maximum cut $(S, T)$ of $H$. Let $G(S,T)$ be the bipartite subgraph with vertex set $V(H)=V \setminus \{v_0\}$ and edge set $E(S,T)$. Observe that $G(S,T)$ is connected due to the maximality of the cut. 

Let $r \in N(v_0)$ and take a spanning tree $T'$ of $G(S,T)$ rooted at $r$. For each node $v \neq r$ in the tree, denote by $par(v)$ its parent in $T'$. We are going to modify $\mu$ on the edges of $T'$ such that $s_{\mu}(v)$ is even if $v \in S$ and odd if $v \in T$. We start with the leafs. For each leaf $\ell$, put $\mu(\{\ell, par(\ell)\}):=3$ if $\ell \in T$. If $\ell \in S$, we keep $\mu(\{\ell, par(\ell)\}) = 2$. Then indeed, $s_{\mu}(\ell)$ is even if and only if $\ell \in S$.

Afterwards, we iterate the idea level by level towards root $r$, processing each internal node only after all its child nodes have been handled: We assign to each tree-edge $\{v, par(v)\}$ weight either $2$ or $3$ such that the parity modulo $2$ of $s_{\mu}(v)$ becomes correct. Finally, we assign to the edge $e_0 = \{v_0, r\}$ weight either $1$ or $2$ such that the parity of $s_{\mu}(r)$ becomes correct as well.

So far, neighboring vertices on different sides of the cut $C$ receive different weighted degrees. We need to ensure that the same happens for neighbors on the \emph{same} side of $C$ as well, and we should not forget $v_0$. The plan is to give each vertex $v \in V$ a designated ``color'' $f(v)$ such that neighboring vertices always receive different colors. Afterwards, from $\mu$ we construct a new edge-weighting $\omega$ such that under $\omega$, indeed each vertex $v$ obtains weighted degree $f(v)$. We start with $v_0$ and its neighborhood. Let $N(v_0)=\{v_1, \ldots, v_m\} \subseteq V(H)$, with arbitrary order. We assign to each $v_i \in N(v_0)$ values $k(v_i)$ and $g(v_i)$ as follows. We start with $k(v_1):=0$ and $g(v_1):=s_{\mu}(v_1)$. For $i > 1$, choose $k(v_i) \in \NN_0$ minimal such that $g(v_i):= s_{\mu}(v_i)+2k(v_i)$ is different from $g(v_j)$ for all $j<i$ with $\{v_i, v_j\} \in E$. If $v_i$ has no such neighbors, use $k(v_i):=0$ and $g(v_i):=s_{\mu}(v_i)$. A vertex $v_i \in S$ thus has at least $k_i$ neighbors in $S$ with smaller index (and the same is obviously true for $T$). For the sake of completeness, set $g(v_0):=s_{\mu}(v_0)$. 

Assume first that $\deg(v_0)>1$. We apply Lemma~\ref{lemma:singlevertex} to the induced subgraph $G[N(v_0) \cup \{v_0\}]$, which is possible since $g$ as defined above indeed satisfies the precondition. The lemma yields a function $h: E(N(v_0) \cup \{v_0\}) \rightarrow \{0,1,2\}$, which we use as follows.

First, we consider the edge weighting. So far, edges incident to $v_0$ have weight either $1$ or $2$. All other edges have weight $2$, except some cut edges $e \in E(S,T)$ with $\mu(e)=3$. For each edge $e \in E(N(v_0) \cup \{v_0\})$, we set $\omega(e) := \mu(e) + h(e)$. For all other edges, we put $\omega(e) := \mu(e)$. Then, edges $e$ incident to $v_0$ satisfy $\omega(e) \le 2 + h(e) \le 4$. Regarding the cut edges, recall that for $u \in S$ and $v \in T$, $s_{\mu}(u)+s_{\mu}(v)$ is odd. Hence, by property (ii) of Lemma~\ref{lemma:singlevertex}, $h$ vanishes on cut edges, implying that we still have $\omega(e)=\mu(e) \in \{2,3\}$ if $e$ is a cut edge of $C$. Finally, if $e \in E(S) \cup E(T)$, then $h(e) \in \{0,1\}$ by (i), implying $\omega(e) \in \{2,3\}$ as well. Thus, on edges not incident to $v_0$, we can further increase or decrease the weighting $\omega$ by $1$ later on.

Second, we assign to each node $v \in N(v_0) \cup \{v_0\}$ the designated color $f(v) := g(v)+s_h(v)$. By property (iii) of Lemma~\ref{lemma:singlevertex}, we preserve parities, i.e., $f$ is even-valued on $S$ and odd-valued on $T$. By (iv), indeed neighboring nodes receive different designated colors. Furthermore, $f(v_0)$ already coincides with $s_{\omega}(v_0)$, and for all $v \in N(v_0)$, we have 
$$f(v)-s_{\omega}(v)=g(v)+s_h(v)-s_{\omega}(v)=2k(v).$$

In the special case $\deg(v_0)=1$, $r=v_1$ is the only neighbor of $v_0$. Here, we directly put $\omega \equiv \mu$, and then set $f(v_0):=s_{\omega}(v_0)=\mu(e_0)$ and $f(v_1):=s_{\omega}(v_1)$. Since $|V|\ge 3$ and $G$ is connected, $v_1$ has at least one other incident edge in addition to $e_0$. Therefore, $s_{\omega}(v_1) > s_{\omega}(v_0)$ and $f(v_1)>f(v_0)$. Clearly, for each edge $e \neq e_0$ we again have $\omega(e) \in \{2,3\}$.

We now turn to the remaining set of vertices $V \setminus (N(v_0) \cup \{v_0\}) := \{v_{m+1}, \ldots, v_{n-1}\}$, which didn't yet receive a designated color. Similarly as before, for each $v_i$, put $f(v_i)=s_{\omega}(v_i)+2k_i$, where $k_i \ge 0$ is the minimal integer such that $f(v_i)$ differs from all $f(v_j)$ for all its neighbors $v_j$ with $1 \le j < i$. Hence, for each vertex $v_i \neq v_0$, we ensured $f(v_i)-s_{\omega}(v_i)=2k(v_i)$. Moreover, any two neighbors of the graph already have different designated colors. For later reference, denote by $t(v_i):=f(v_i)-2k_i$ the current weighted degree of $v_i$ under $\omega$. 

To actually achieve the desired colors, the weighted degree $s_{\omega}(v_i)$ of each node $v_i$ should further increase by exactly $2k_i$. In order to solve this task, we construct a subset $F \subseteq E[S] \cup E[T]$ and an orientation $\sigma$ of $F$ as follows. For each vertex $v_i \in S$, choose $k_i$ neighbors $v_j \in S$ with smaller index (i.e., $1 \le j<i$), add $\{v_i, v_j\}$ to $F$, and add the orientation $(v_i, v_j)$ to $\sigma$. For each vertex $v_i \in T$, choose $k_i$ neighbors $v_j \in T$ with $1 \le j<i$, add $\{v_i, v_j\}$ to $F$, but add orientation $(v_j, v_i)$ to $\sigma$ (mind the asymmetry compared to side $S$!).

Then by applying Lemma~\ref{lemma:flow} to the reduced graph $H$, there is an $s$-$t$-flow of size $|F|$ in the auxiliary multigraph $H_{C,F,\sigma}$. As all edges have capacity $1$, there are $f$ edge-disjoint $s$-$t$-paths in $H_{C,F,\sigma}$. Consider such a path $p=(s, u_1, \ldots, u_m, t)$, and let $p' = \{u_1, \ldots, u_m\}$ be its induced, undirected subpath in the bipartite graph $G[S,T]$. Unless $u_1 = u_m$ (which happens when $p'$ is an empty path), we modify the weighting $\omega$ of each edge $\{u_i, u_{i+1}\} \in p'$ as follows: increase its weight by $1$ if $u_i \in S$, and decrease the weight by $1$ if $u_i \in T$. In other words, we alternately increase or decrease the edge weights along the path. The weighted degrees of the internal nodes $u_2, \ldots, u_{m-1}$ thereby do not change, in contrast to those of the endpoints $u_1$ and $u_m$. The weighted degree of $u_1$ increases by $1$, if $u_1 \in S$, and decreases by $1$, if $u_1 \in T$. Regarding $u_m$, its weighted degree increases by $1$, if $u_m \in T$, and decreases by $1$, if $u_m \in S$. When $u_1=u_m$, there is no change on the weighted degree of this node.

By construction of $H_{C,F,\sigma}$, each edge of $F$ led to exactly one arc incident to $s$ and one arc incident to $t$. Thus, for each path $p=(s,u_1, \ldots, u_m, t)$ of the provided edge-disjoint $s$-$t$-paths, there are two uniquely identified $F$-incidences: an edge $f^+ = \{u_1, w_1\} \in F$ with $(u_1, w_1) \in \sigma$, leading to the arc $(s,u_1) \in p$, and an edge $f^-=\{w_m,u_m\} \in F$ with $(w_m, u_m) \in \sigma$, leading to the arc $(u_m, t) \in p$. Note that $f^+=f^-=\{u_1,u_m\}$ is possible. Vice versa, as we found $|F|$ edge-disjoint $s$-$t$-paths in the auxiliary network $H_{C,F,\sigma}$, for each $f=\{u^+,u^-\} \in F$ with $(u^+,u^-)\in\sigma$, there are uniquely identified paths starting with $(s,u^+)$ and ending with $(u^-,t)$. 

We repeat the described modification on $\omega$ for all $|F|$ paths provided by Lemma~\ref{lemma:flow}. Summing up the changes on $\omega$ caused by each path $p'$, for $v \in S$, the freshly updated edge-weighting $\omega$ satisfies
$$s_{\omega}(v)- t(v) = |\{w:(v,w) \in \sigma\}| - |\{w:(w,v) \in \sigma\}|,$$
whereas for $v \in T$, 
$$s_{\omega}(v)- t(v) = |\{w:(w,v) \in \sigma\}| - |\{w:(v,w) \in \sigma\}|.$$

Finally, as a last modification step, we increase the weighting $\omega$ on each edge in $F$ by $1$, obtaining
$$s_{\omega}(v)- t(v) = 2|\{w:(v, w) \in \sigma\}|=2k_i$$
for all $v \in S$, and 
$$s_{\omega}(v)- t(v) = 2|\{w:(w, v) \in \sigma\}|=2k_i$$
for all $v \in T$. We conclude that each vertex $v_i$ obtained weighted degree $t(v_i)+2k_i=f(v_i)$, thus indeed, the constructed edge-weighting $\omega$ gives rise to a proper vertex-coloring of $G$.
\end{proof}

\begin{proof}[Proof of Lemma~\ref{lemma:flow}]
Let $k := |F|$ and consider the network $H := G_{C, F, \sigma}$ with auxiliary vertices $s$ and $t$. Assume by contradiction that there exists no $s$-$t$-flow of value $k$ in $G_{C, F, \sigma}$. Then, by the standard \emph{max-flow min-cut theorem} \cite[e.g.]{ford2010flows}, there exists a cut $C'=(A_H, B_H)$ of size $\ell<k$, where $s \in A_H$ and $t \in B_H$. Observe that $A_G := A_H \setminus \{s\}$ and $B_G := B_H \setminus \{t\}$ are both subsets of the vertex set $V$ of the original, undirected graph $G=(V,E)$.

By step (iii) of the lemma statement, for each edge $f \in F$, in $H$ there are two uniquely defined arcs outgoing at $s$ and incoming at $t$. Let $F' := \{\{u,v\} \in F: u \in A_G, v \in B_G\}$ and remark that for each $f=\{u,v\} \in F \setminus F'$ with $(u,v) \in \sigma$, one of its two identified arcs $(s,u)$ and $(v,t)$ is in the cut $C'$. Next, let $E_1 := E(S \cap A_G, T \cap B_G) \cup E(S \cap B_G, T \cap A_G)$
and notice that for each edge $e = \{u,v\} \in E_1$, either $(u,v)$ or $(v,u)$ is contained in $C'$ as well. It follows
$$|E_1| \le \ell - |F\setminus F'| = \ell - (k-|F'|) < |F'|.$$

Finally, consider the cut 
$$C'' := ((S \cap A_G) \cup (T \cap B_G),(S \cap B_G) \cup (T \cap A_G))$$ 
of the original graph $G$. Let $E_2 := E(S \cap A_G, S \cap B_G) \cup E(T \cap A_G, T \cap B_G)$, 
and observe that $F' \subseteq E_2$. Putting everything together, we deduce
$$|C''| = |C| - |E_1|+|E_2| > |C|-|F'|+|F'|=|C|.$$
We see that in $G$, $C''$ would be a larger cut than $C$, contradicting the maximality of cut $C$. So, in $H$ there exists an $s$-$t$-flow of value $k$.
\end{proof}

\begin{proof}[Proof of Lemma~\ref{lemma:singlevertex}]In order to prove the statement, we do a case analysis to find a suitable function $h:E\rightarrow \{0,1,2\}$ that satisfies (i)-(iv). Whenever such a function $h$ is constructed, for each $v \in V$, define $f_h(v) := g(v)+s_h(v)$. Then (iv) is fulfilled if and only if $f_h$ is a proper vertex coloring of $G$.

We start with the case $g(v_0) \neq g(v)$ for all vertices $v \neq v_0$, where we can set $h \equiv 0$. Then (i)-(iv) are all met. Otherwise, let 
$$V' := \{v \neq v_0: g(v) - g(v_0) \equiv 0 \text{ mod } 2\}=\{v_1, \ldots, v_{m'}\},$$ 
and assume w.l.o.g.\ that $g(v_1) \le \ldots \le g(v_{m'})$. Suppose first that $m'=|V'|=1$. Because we already handled the case $g(v_0) \neq g(v_1)$, we can assume $g(v_0)=g(v_1)$. Take $u \in  N(v_0) \setminus V'$ such that $g(u)$ is maximal. Put $h(\{v_0,u\}) := 2$ and $h(e):=0$ for any other edge $e \neq \{v_0,u\}$. We obtain $f_h(v_1)< f_h(v_0)$. Since for any $w \in V \setminus \{u, v_0, v_1\}$, it holds
$$f_h(u) = g(u)+2 \ge g(w)+2 > g(w)= f_h(w),$$ 
(iv) is achieved. The other properties clearly hold as well.

For the remaining proof, we can assume $m' \ge 2$. For all $j \in \{0, \ldots, m'\}$, define the function 
$$
h_j: E \rightarrow \NN_0,\hspace{0.2cm} e \mapsto
\begin{cases}
2, & \text{if } e=\{v_0,v_i\} \text{ for some } j<i \le m',\\
0, & \text{otherwise}.
\end{cases}
$$
The plan is now to find a suitable function $h_j$ for as many cases as possible. Clearly, each $h_j$ directly fulfills (i)-(iii). Regarding (iv), we claim that it is sufficient to verify the property only for edges incident to $v_0$. Indeed, by construction, on $V'$ $g(u)>g(v)$ directly implies $f_{h_j}(u) > f_{h_j}(v)$. Hence, $f_{h_j}$ would already properly color $V' \cup \{v_0\}$. But $f_{h_j}$ also preservers the parities mod $2$ of $g$, and $h_j$ vanishes on $E \setminus E(V' \cup \{v_0\})$, so the proper coloring of the remaining vertices is inherited from $g$.

Let $x \ge 1$ be the smallest integer such that $g(v_0)+2x$ differs from all $g(v_1), \ldots, g(v_{m'})$, and let $i' \le m'$ be maximal such that $g(v_{i'}) < g(v_0)+2x$. Because $g(v_0)=g(v_i)$ for some $v_i$, $i'$ is well-defined. First consider the case $i' \le m'-x$. Here, we set $h \equiv h_{m'-x}$. Then, for $i >m'-x \ge i'$, we have $f_h(v_i) >g(v_i) \ge g(v_0)+2x$, whereas for $i \le m'-x$, it holds $f_h(v_i)=g(v_i) \neq g(v_0)+2x$. Hence indeed, $f_h(v_0) = g(v_0)+2x$ is different from $f_h(v_1), \ldots, f_h(v_m)$.

Next, consider the case $i' > m'-x$ and $x<m'$, where we put $h \equiv h_{m'-x-1}$. For $i \ge m'-x$ we then have $f_h(v_i)=g(v_i)+2 \neq g(v_0)+2(x+1)$, whereas for all $i < m'-x$ it holds $$f_h(v_i)=g(v_i) \le g(v_{i'}) < g(v_0)+2x.$$ 
Thus, for all $v_i \in V'$, it holds $f_h(v_0) = g(v_0)+2(x+1)\neq f_h(v_i)$, and (iv) is again fulfilled.

It remains the case $x=m'$. Here, for each $0\le y < m'$, the value $g(v_0)+2y$ is attained by one $g(v_i)$. So we have $g(v_i)=g(v_0)+2(i-1)$ for all $v_i \in V'$. We distinguish two subcases. If $m'$ is even, we can use $h \equiv h_{m'/2}$. For $i \le m'/2$, it then holds $f_h(v_i)=g(v_i) \le g(v_0)+m'-2$, whereas, for $i > m'/2$, $f_h(v_i)=g(v_i)+2 \ge g(v_0)+m'+2$. Since $f_h(v_0)=g(v_0)+m'$, (iv) is again achieved. 

On the other hand, if $m'$ is odd, then $m' \ge 3$. This situation is a bit inconvenient, because none of the functions $h_j$ can be used. Instead, let $z := \tfrac{m'+3}{2} \le m'$. Put $h(\{v_0, v_i\}) := 2$ for all $i > z$, and $h(\{v_0, v_i\}):=0$ for all $i < z-1$. Moreover, let $h(e):=0$ whenever $e \notin E(V' \cup \{v_0\})$, so (ii) is already satisfied. We want to achieve $s_h(v_0)=m'-1$, but need to be careful to satisfy (iii) and (iv) at the same time. If $v_z$ and $v_{z-1}$ do not share an edge, put $h(\{v_0, v_{z-1}\}):=2$, $h(\{v_0, v_z\}):=0$, and $h(e):=0$ for all $e \in E(N(v_0))$. Then 
$$f_h(v_z)=f_h(v_{z-1})=g(v_0)+2(z-1)=g(v_0)+m'+1,$$ 
which is fine regarding (iv), as the two nodes are not neighbors. 

Vice versa, if the edge $e' := \{v_z, v_{z-1}\}$ is present in $E$, for each edge $f$ of the triangle $\{v_0, v_{z-1},v_z\}$, put $h(f):=1$. For all edges $e \in E(V') \setminus \{e'\}$, set $h(e):=0$, yielding 
$$f_h(v_z)=g(v_0)+2z = g(v_0)+m'+3$$ 
and 
$$f_h(v_{z-1})=g(v_0)+2(z-1)=g(v_0)+m'+1.$$ 
In both subcases, $f_h(v_0)=g(v_0)+m'-1$ by construction. Moreover, for each $v_i$ with $i \ge z+1$, it holds 
$$f_h(v_i) = g(v_i)+2 = g(v_0)+2i \ge g(v_0)+m'+5,$$ 
whereas for $i < z-1$, it holds
$$f_h(v_i)=g(v_i) =g(v_0)+2(i-1) \le g(v_0)+2(z-3) =g(v_0)+m'-3.$$
We conclude that $f_h$ properly colors $V' \cup \{v_0\}$. Properties (i)-(iii) are clearly achieved with $h$ in both subcases. By the same argument as above for the functions $h_j$, $f_h$ then properly colors the entire set $V$.
\end{proof}

\bibliographystyle{acm}
\bibliography{references}

\end{document}